\documentclass[a4paper]{amsart}
\usepackage{amssymb, amsthm, amsmath, pifont} 
\usepackage[usenames]{color}
\usepackage{amsfonts}
\usepackage{enumerate}
\usepackage{pdfpages}
\usepackage[all]{xypic}
\usepackage{graphicx}
\usepackage{psfrag}
\usepackage{verbatim}
\usepackage{pstricks}
\usepackage{mathrsfs}
\usepackage[all,knot]{xy}
\usepackage{todonotes}
\usepackage{tikz}
\usetikzlibrary{decorations}
\usetikzlibrary{decorations.pathmorphing}
\newtheorem{theorem}{Theorem}

\newtheorem{corollary}[theorem]{Corollary}

\newtheorem{prop}[theorem]{Proposition}
\newtheorem{proposition}[theorem]{Proposition}

\newtheorem{lemma}[theorem]{Lemma}
\theoremstyle{definition}
\newtheorem{question}[theorem]{Question}
\newtheorem{example}[theorem]{Example}
\newtheorem{definition}[theorem]{Definition}
\newtheorem*{intdef}{Definition}
\theoremstyle{remark}
\newtheorem{remark}[theorem]{Remark}
\newtheorem{claim}[theorem]{Claim}
\newtheorem{obs}[theorem]{Observation}

\newtheorem*{acknowledgments}{Acknowledgments}

\def\et{\;\mbox{and}\;}

\def\for{\quad\mbox{for }}

\def\et{\quad\mbox{and}\quad}

\def\wt{\widetilde}

\def\epsilon{\varepsilon}

\def\R{\mathbb{R}}
\def\Z{\mathbb{Z}}
\def\C{\mathbb{C}}

\title{Up to topological concordance links are strongly quasipositive}

\author{Maciej Borodzik}
\address{Institute of Mathematics, University of Warsaw, ul. Banacha 2,
02-097 Warsaw, Poland}
\email{mcboro@mimuw.edu.pl}
\author{Peter Feller}
\address{Department of Mathematics, ETH Z\"urich, R\"amistrasse 101, 8092 Z\"urich, Switzerland}
\email{peter.feller@math.ch}
\begin{document}
\begin{abstract}
We generalize an algorithm of Rudolph to establish that every link is topologically concordant to a strongly quasipositive link.
\end{abstract}
\maketitle
\section{Introduction}
\subsection{Statement of results}
A braid or $n$-braid in Artin's braid group on $n$ strands, i.e. \[B_n=\left\langle a_1,\ldots,a_{n-1}\;\middle|\;a_ia_j=a_ja_i\text{ for }|i-j|\geq 2,a_ia_{i+1}a_i=a_{i+1}a_ia_{i+1}\right\rangle\quad\text{\cite{Artin_TheorieDerZoepfe},}\]
is called \emph{quasipositive} if it is a product of conjugates of positive Artin generators $a_i$ of $B_n$.
A~braid is called \emph{strongly quasipositive} if it is a product of braids $a_{i,j}$, where
\begin{equation}\label{eq:positive_braid_word}\left\{a_{i,j}=(a_ia_{i+1}\cdots a_{j-1})a_j(a_ia_{i+1}\cdots a_{j-1})^{-1}|\for 1\leq i\leq j\leq n\right\}.\end{equation}
A \emph{link} $L$---an oriented smooth 1-dimensional submanifold of $S^3$---is called \emph{quasipositive} (respectively: \emph{strongly quasipositive}) if it arises as the closure of a quasipositive (respectively: strongly quasipositive)
braid. Clearly, any strongly quasipositive link is also quasipositive. We refer to~\cite{Rudolph_handbook} for a resource on quasipositive links.

It was proved by Boileau and Orevkov that a link is quasipositive if and only if it can be realized as the transverse intersection of a complex curve in $\C^2$ with the unit 3-sphere $S^3\subset\C^2$;
see~\cite{BoileauOrevkov_QuasiPositivite}. Given this result, quasipositive links form a bridge between knot theory and the theory of complex curves in $\C^2$. This in particular
suggests that
quasipositive knots should have special properties in the concordance group.
In this light, the main result of the paper seems counterintuitive.
%
\begin{theorem}\label{thm:main1}
Every link $L$ is topologically concordant to a strongly quasipositive link $L'$.
\end{theorem}
In fact, we establish 
the following more precise, but also more technical result.
\begin{theorem}\label{thm:main2}
Let $L\subset S^3$ be a link with $m$ components. There is a strongly quasipositive $m$-component link $L'$ and a topologically locally flat
surface $A\subset S^3\times[0,1]$ which is a disjoint
union of $m$ annuli such that
\begin{itemize}
  \item $\partial A=L\times\{0\}\sqcup L'\times\{1\}\subset S^3\times[0,1]$;
  \item the inclusion induced map $\pi_1(S^3\setminus L)\to \pi_1(S^3\times[0,1]\setminus A)$ is an isomorphism;
  \item the inclusion induced map $\pi_1(S^3\setminus L')\to \pi_1(S^3\times[0,1]\setminus A)$ is a surjection.
\end{itemize}
\end{theorem}
In the above, the surface $A$ is oriented and the induced orientations on $L\times\{0\}$ and $L'\times\{1\}$ are the orientation of $L\times\{0\}$ and the reversed orientation of $L'\times\{1\}$, respectively.
\subsection{Context}\label{sec:overview}

In~\cite{Rudolph_83_ConstructionsOfQP1} Rudolph proved the following result (see also \cite[Proposition~87]{Rudolph_handbook}):
\begin{theorem}\label{thm:rudolph}
  If $F$ is a Seifert surface of an $m$-component link in $S^3$, then there exists a strongly quasipositive $m$-component link $L$ which arises as the boundary of a reembedding of $F$ in $S^3$ that has the same Seifert form as $F$.\qed
\end{theorem}
Theorem~\ref{thm:main2} can be regarded as a generalization of this result; compare also with Remark~\ref{rem:SM}. For example, one of the corollaries of Theorem~\ref{thm:rudolph} is that
every algebraic concordance class of links contains a strongly quasipositive link. Our result is that every topological concordance class of links contains a strongly quasipositive link.

It follows from Theorem~\ref{thm:rudolph} that no link invariant obtained from the S-equivalence class of a Seifert matrix can tell
a strongly quasipositive link from a non strongly quasipositive one; see~\cite[Proposition 87]{Rudolph_handbook}. Our result says that if a link invariant
is equal on topologically concordant links, then this link invariant cannot distinguish strongly quasipositive links from non strongly quasipositive links.
For example we have the following result.
\begin{corollary}\label{cor:cor}
The linking numbers and the Milnor invariants do not distinguish strongly quasipositive links from non quasipositive links: if there is a link $L$ with a given set of Milnor invariants, then the same set can be realized as a set of Milnor invariants
of a strongly quasipositive link $L'$. \qed 
\end{corollary}
\begin{remark}

	The fact that there exist strongly quasipositive links with arbitrary linking numbers follows already from Rudolph's result (Theorem~\ref{thm:rudolph}).
The corollary is a generalization of that fact, because linking numbers are second-order Milnor invariants.
\end{remark}
In \cite{Harvey} Harvey defined a family of concordance invariants $\rho_k$, $k\in\Z_{\ge 0}$, which are an instance of $L^2$-invariants; see \cite[Section 5]{Cochran_Orr_Teichner}.
As these invariants are invariants of topological concordance (see \cite[Corollary 4.3]{Harvey}), we obtain the following result.
\begin{corollary}
  The Harvey's $\rho_k$ invariants do not detect strongly quasipositive knots.\qed
\end{corollary}
Gilmer \cite[Section 2]{Gilmer} interprets Casson-Gordon invariants as a homomorphism from the knot concordance group to a suitably defined Witt group.
In this context we can say that Casson-Gordon invariants do not detect strongly quasipositive knots. Another way of phrasing the failure of
Casson-Gordon invariants to detect strongly quasipositive knots is the following.
\begin{proposition}\label{prop:CG}
  Let $K\subset S^3$ be a knot. Then there exists a strongly quasipositive knot $K'$ such that for any $k>1$ the spaces $\operatorname{Hom}(H_1(\Sigma_k(K);\Z);\C^*)$
  and $\operatorname{Hom}(H_1(\Sigma_k(K');\Z);\C^*)$ are identified and under this identification, for every $\chi\in\operatorname{Hom}(H_1(\Sigma_k(K);\Z);\C^*)$
  that is of prime power order
  we have the equality of Casson--Gordon invariants $\tau(K;\chi)=\tau(K';\chi)$.
\end{proposition}
\begin{proof}
  It follows from the proof of Theorem~\ref{thm:main2} that $K'$ can be obtained by an iterated satellite construction with companion knots that have trivial Alexander polynomial and
  winding number is zero (see Proposition~\ref{prop:is_satellite} below). The statement follows from \cite[Theorem 2]{Litherland}. See~\cite{LivSeifert} for an analogous argument.
\end{proof}
A result analogous to Proposition~\ref{prop:CG} can also be stated in the language of $\eta$-invariants of Levine and Friedl; see \cite{Friedl-eta}, using
\cite[Theorem 4.2]{Friedl-eta} as the main technical result. We do not give a precise statement here.


Given Theorem~\ref{thm:main1} one can ask whether the result can be generalized to smooth concordance. Here the answer is known to be negative by
Rudolph's slice-Bennequin inequality~\cite{rudolph_QPasObstruction}; compare also~\cite{Hedden_Positive,Livingston_Comp,Plamenevskaya_TB_number,Shumakovitch} where the values for the concordance invariants $s$ and $\tau$ are shown to be restricted on quasipositive knots.
In fact, the latter together with the existence of knots $K$ for which $\tau(K)$ and $-s(K)/2$ differ implies the following result:
\begin{proposition}\label{prop:not_smoothly}
  There exists a knot $K$ that is not smoothly concordant to any sum $P\#N$, where $P$ is a quasipositive knot and $N$ is quasinegative (the mirror of $N$ is quasipositive).
\end{proposition}
\begin{proof}
	Let $\tau(K)$ be the Ozsv\'ath--Szab\'o $\tau$ invariant and $s(K)$ the Rasmussen's $s$-invariant. It follows from
	\cite{Hedden_Positive,Livingston_Comp,Plamenevskaya_TB_number,Shumakovitch} that if $K$ is quasipositive, then $\tau(K)=g_4(K)=-s(K)/2$.
	In particular, the map $K\mapsto \sigma(K):=\tau(K)-(-s(K)/2)$ vanishes on all quasipositive knots. It is also an additive concordance invariant, so it vanishes
  on the whole smooth concordance subgroup generated by quasipositive knots. However, it was shown by Hedden and Ording~\cite{Hedden_Ording} that $\sigma$ does not vanish on all knots.
\end{proof}
\begin{remark}
  One can use any difference of any two different slice-torus invariants (i.e.~group homomorphisms from the smooth concordance group to $\Z$ that equal $g_4$ on positive torus knots and are less than or equal to $g_4$ on all knots; compare Lewark~\cite{Lewark_advances}) instead of $\tau(K)-(-s(K)/2)$ in the above proof.
\end{remark}
The contrast between topological and smooth concordance of strongly quasipositive knots can also be seen by comparing the position of strongly
quasipositive knots in various filtrations of the knot concordance group.
\begin{proposition}\label{prop:chh_filter}\
\begin{itemize}
 \item Consider the solvable filtration $\mathcal{F}\supset \mathcal{F}_0\supset \mathcal{F}_{0.5}\supset \ldots$ of the topological concordance group $\mathcal{F}$ introduced by Cochran, Orr, and Teichner~\cite{Cochran_Orr_Teichner}.
For any integer $n\geq0$, there exists a strongly quasipositive knot $K$ that is $n$-solvable (i.e.~$[K]\in \mathcal{F}_{n}$) but not $(n+0.5)$-solvable (i.e.~$[K]\notin \mathcal{F}_{n+0.5}$).
 \item Consider the bipolar filtration of Cochran, Harvey and Horn~\cite{Cochran_Harvey_Horn}. The only strongly quasipositive knot $K$ that belongs to $\mathcal{B}_0=\mathcal{P}_0\cap\mathcal{N}_0$ (notation from~\cite{Cochran_Harvey_Horn}) is trivial.
	 \end{itemize}
 \end{proposition}
\begin{proof}
	The first part is immediate from Theorem~\ref{thm:main1} and the fact that there are knots that are $n$-solvable but not $(n+0.5)$-solvable~\cite{CochranTeichner_07}. For the second part, note that
  if $K\in\mathcal{B}_0$ then, by~\cite[Corollary 4.9]{Cochran_Harvey_Horn} we have $\tau(K)=0$. If $K$ is strongly quasipositive, then by~\cite{Livingston_Comp} we have
  $g_3(K)=\tau(K)$. But the only knot with vanishing three-genus is the unknot.
\end{proof}
The exact position of strongly quasipositive knots in the smooth concordance group remains unknown. For example, Baker conjectures that concordant fibred strongly quasipositive knots are isotopic~\cite{Baker}.

We conclude this section by
noting the substantial difference between strongly quasipositive and quasipositive knots. While
there exist non-trivial topologically slice strongly quasipositive knots---one of the simplest instances being the Whitehead
double of the trefoil; see~\cite[Fig.~1]{Rudolph_84_SomeTopLocFlatSurf} or~\cite[Lemma~2]{rudolph_QPasObstruction}---the only smoothly slice strongly quasipositive knot is, by the slice-Bennequin inequality~\cite{Rudolph_83_ConstructionsOfQP1}, the unknot.
In contrast, Rudolph~\cite{Rudolph_83_ConstructionsOfQP1} constructed many
non-trivial smoothly slice quasipositive knots. This distinction is also stressed in Remark~\ref{rem:qpbandsvssqpbands}.

\subsection{Outline of Proof}
Theorem~\ref{thm:main1} follows immediately from Theorem~\ref{thm:main2}.
We now outline the proof of Theorem~\ref{thm:main2}.
The proof consists of a generalization of the proof of Theorem~\ref{thm:rudolph} combined with Freedman's disk theorem and the behavior of concordance under satellite operations.

Let $w$ be an $n$-braid with closure the link $L$. Write $w$ as a product of braids $a_{i,j}$ (defined in \eqref{eq:positive_braid_word})
and $a_{i,j}^{-1}$. 
Construct for this product the associated canonical Seifert surface $F$. It consists of $n$ disks 
connected with bands: an element $a_{i,j}^{\pm 1}$ corresponds
to a band connecting the $i$-th disk with the $j$-th one and the sign $\pm 1$ corresponds to the sign of twist.
As in~\cite{Rudolph_83_ConstructionsOfQP1} we replace each negative band by a band obtained by thickening a strongly quasipositive knot with a zero framing (Rudolph used a zero-framed positive trefoil knot, which suffices to get Theorem~\ref{thm:rudolph}).
A careful argument gives us control of the strong quasipositivity
of the resulting surface independently of what strongly quasipositive knot this is done with. More precisely, we
establish the following proposition which constitutes the bulk of our proof.

\begin{prop}\label{prop:main}
Let $F$ be a Seifert surface for some link $L\subset S^3$. There exist a finite number of closed bands $S_1, \ldots , S_t$ in $F$ such that the following holds. For any choice of $t$ non-trivial strongly quasipositive knots $K_1$,\dots, $K_t$, the surface $P$ obtained from $F$ by tying a zero-framed $K_\ell$ into the band $S_\ell$ is strongly quasipositive.

If $F$ is the canonical Seifert surface associated with a braid $w$ given as a product $\prod a_{i_k,j_k}^{\epsilon_k}$ with $\epsilon_k=\pm 1$ and $a_{i,j}$ as in
\eqref{eq:positive_braid_word}, then $t$ can be taken to be the number of indices $k$ for which $\epsilon_k$ is negative.
\end{prop}
Theorem~\ref{thm:main2} follows from Proposition~\ref{prop:main} by choosing the $K_\ell$ to be strongly quasipositive knots with Alexander polynomial one, invoking a consequence of Freedman's disk embedding theorem to establish that $K_\ell$ is topologically slice~\cite[Theorem~1.13]{Freedman_82_TheTopOfFour-dimensionalManifolds}, and using that a satellite link is concordant to its pattern if the companion is a slice knot.
\begin{remark}\label{rem:SM}
  The procedure described in Proposition~\ref{prop:main} does not change the Seifert form associate to the surface $F$. As a consequence,
  our construction of the link $L'$ from Theorem~\ref{thm:main2} is such that $L$ and $L'$ come with Seifert surfaces with canonically identified Seifert forms.
\end{remark}

\begin{acknowledgments}
The first author is supported by the Polish National Science Center grant 2016/22/E/ST1/00040. The authors express their gratitude to Sebastian Baader, Michel Boileau, and Wojciech Politarczyk for stimulating discussions.
We are indebted to Chuck Livingston, Mark Powell and Lee Rudolph for their comments on the preliminary version of the manuscript.
\end{acknowledgments}

\section{Proof of Proposition~\ref{prop:main}}\label{sec:proofOfProp}
In this section, we first recall details on the notions used in Proposition~\ref{prop:main} and then provide the proof of Proposition~\ref{prop:main}.

Let $F$ be a \emph{Seifert surface}---a smooth oriented compact surface in $S^3$ with no closed components. 
Up to isotopy $F$ arises as the canonical surface associated with a braid $w\in B_n$ for some $n$ that is written as a product of generators
\begin{equation}\label{eq:wprod}
w=\prod a_{i_k,j_k}^{\epsilon_k},
\end{equation}
where $a_{i,j}$ are as in~\eqref{eq:positive_braid_word}, $i_k,j_k$ are some index pairs and $\epsilon_k=\pm 1$;
see~\cite[Proposition~1]{Rudolph_83_ConstructionsOfQP1}. We refer to such a $w$, a braid together with a choice of product as in~\eqref{eq:wprod}, as \emph{braid word} or \emph{$n$-braid word}. If all $\epsilon_k$ are $1$, we refer to $w$ as a \emph{strongly quasipositive braid word}.

Here, the \emph{canonical surface} associated to such a $n$-braid word $w$ 
is the 
Seifert surface given by $n$ discs (one corresponding to each strand) and bands connecting the discs (one for each $a_{i,j}^{\pm 1}$).
The canonical surface is illustrated for an example in Figure~\ref{fig:QPA}. The vertical strands correspond to disks and the horizontal hooked strands
are the bands. The upward hook represents a positively twisted band, while the downward hook, like the bottom hook in Figure~\ref{fig:w},
represents a negatively twisted band; in fact, the canonical surface arises as the blackboard framed thickening of the depicted graph. We refer to Rudolph's account in~\cite[$\S$2]{Rudolph_83_BraidedSurfaces} and~\cite[$\S$1]{Rudolph_92_ConstructionsIII} for more details on canonical Seifert surfaces. 
A Seifert surface is said to be a \emph{quasipositive surface}, if it arises as the canonical surface associated to a strongly quasipositive braid word. 
Note that strongly quasipositive links are the links that arise as the boundary of a quasipositive surface.

We call a subset $S$ of a Seifert surface $F$ a \emph{closed band} if $S$ arises as the intersection of $F$ with a cylinder $Z\subset S^3$ as follows: there is a orientation preserving diffeomorphism $\Psi$ of triples of manifolds with corners
\[\Psi\colon(Z, Z\cap F=S, Z\cap \partial F)\to(D^2\times[0,1],[-\tfrac{1}{2},\tfrac{1}{2}]\times[0,1],\{-\tfrac{1}{2},\tfrac{1}{2}\}\times[0,1]),\]
where $D^2\subset \C$ denotes the unit disk and $[-\tfrac{1}{2},\tfrac{1}{2}]$ denotes the straight line segment connecting $-\tfrac{1}{2}$ and $\tfrac{1}{2}$ in $D^2$; see left-hand side of Figure~\ref{fig:tangle}. (Here standard orientations on $D^2\times[0,1]$ and $[-\tfrac{1}{2},\tfrac{1}{2}]\times[0,1]$ are chosen, e.g.~induced from the standard orientations of $\R^3$ and $\R^2$, respectively, and the orientation on $\{-\tfrac{1}{2},\tfrac{1}{2}\}\times[0,1]$ is the one induced from $[-\tfrac{1}{2},\tfrac{1}{2}]\times[0,1]$.)
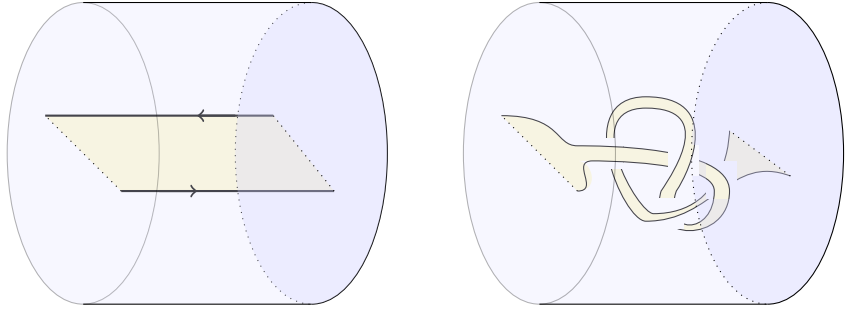
\begin{figure}[h]
\begin{tikzpicture}
  \begin{scope}[xshift=-3cm]
    \fill[yellow!20,draw=none] (0,0) -- (0.5,-0.5) -- (3.3,-0.5)--  (2.5,0.5) -- (-0.5,0.5)--  (0,0);
    \draw[thick] (0,0) (0.5,-0.5) -- (3.3,-0.5) (2.5,0.5) -- (-0.5,0.5) (0,0);
    \draw[thick,->] (1,-0.5) -- (1.5,-0.5);
    \draw[thick,->] (2,0.5) -- (1.5,0.5);
    \draw[fill=blue!10,opacity=0.3] (0,0) ellipse (1 and 2);
    \draw[fill=blue!20,opacity=0.3,draw=none] (3,0) ellipse (1 and 2);
    \draw[dotted] (3,2) arc [start angle=90, delta angle=180, x radius=1, y radius=2];
    \draw (3,2) arc [start angle=90, delta angle=-180, x radius=1, y radius=2];
    \fill[draw=none, blue!10,opacity=0.3] (0,2) arc [start angle=90, delta angle=-180,x radius=1,y radius=2] -- (3,-2) arc [start angle=270,delta angle=180, x radius=1,y radius=2] -- (0,2);
    \draw (0,2) -- (3,2);
    \draw (0,-2) -- (3,-2);
    \draw[thin,dotted] (0.5,-0.5) -- (-0.5,0.5);
    \draw[thin,dotted] (3.3,-0.5) -- (2.5,0.5);
\end{scope}

  \begin{scope}[xshift=3cm]
    \fill[yellow!20,draw=none]
    (-0.5,0.5) .. controls (0.3,0.5) and (0.2,0.1) .. (0.5,0.1) .. controls (0.8,0.1) and (2.5,0.1) .. (2.5,-0.5) .. controls (2.5,-1.2) and (1.4,-1.2) .. (1.6,-0.5) ..
    controls (1.8,-0.2) and (2.3,0.6) .. (1.4,0.6) .. controls (0.5,0.6) and (1.3,-0.8) .. (1.5,-0.8) .. controls (2.5,-0.8) and (2.5,0.3) .. (2.5,0.3) --
    (3.3,-0.3) .. controls (2.5,0.0) and (2.3,-0.9) .. (1.4,-0.9) .. controls (1.1,-0.9) and (0.4,0.75) .. (1.4,0.75) .. controls (2.4,0.75) and (2.0,-0.2) ..
    (1.8,-0.5) .. controls (1.6,-1.1) and (2.3,-1.1) .. (2.3,-0.5) .. controls (2.3,-0.1) and (0.8,-0.1) .. (0.5,-0.1) .. controls (0.7,-0.1) and (0.8,-0.5) .. (0.5,-0.5) -- (-0.5,0.5);
    \draw (-0.5,0.5) .. controls (0.3,0.5) and (0.2,0.1) .. (0.5,0.1) .. controls (0.8,0.1) and (2.5,0.1) .. (2.5,-0.5) .. controls (2.5,-1.2) and (1.4,-1.2) .. (1.6,-0.5) ..
    controls (1.8,-0.2) and (2.3,0.6) .. (1.4,0.6) .. controls (0.5,0.6) and (1.3,-0.8) .. (1.5,-0.8) .. controls (2.5,-0.8) and (2.5,0.3) .. (2.5,0.3);
    \draw (0.5,-0.5) .. controls (0.8,-0.5) and (0.4,-0.1) .. (0.7,-0.1) .. controls (0.8,-0.1) and (2.3,-0.1) .. (2.3,-0.5) .. controls (2.3,-1.1) and (1.6,-1.1) .. (1.8,-0.5) ..
    controls (2.0,-0.2) and (2.4,0.75) .. (1.4,0.75) .. controls (0.4,0.75) and (1.1,-0.9) .. (1.4,-0.9) .. controls (2.3,-0.9) and (2.5,0.0) .. (3.3,-0.3);

    \draw[fill=blue!20,opacity=0.3,draw=none] (3,0) ellipse (1 and 2);
    \fill[white,draw=none] (0.8,-0.2) rectangle (1.1,0.2);
    \fill[white,draw=none] (1.7,-0.3) rectangle (2.1,0.1);
    \fill[white,draw=none] (2.2,-0.6) rectangle (2.6,-0.1);
    \fill[blue!20,opacity=0.3,draw=none] (2.2,-0.6) rectangle (2.6,-0.1);
    \fill[white,draw=none] (1.5,-1) rectangle (1.9,-0.6);
    \begin{scope}
      \clip (0.8,-0.2) rectangle (1.1,0.2);
    \fill[yellow!20,draw=none]
    (-0.5,0.5) .. controls (0.3,0.5) and (0.2,0.1) .. (0.5,0.1) .. controls (0.8,0.1) and (2.5,0.1) .. (2.5,-0.5)--
    (2.3,-0.5) .. controls (2.3,-0.1) and (0.8,-0.1) .. (0.5,-0.1) .. controls (0.7,-0.1) and (0.8,-0.5) .. (0.5,-0.5) -- (-0.5,0.5);
    \draw (-0.5,0.5) .. controls (0.3,0.5) and (0.2,0.1) .. (0.5,0.1) .. controls (0.8,0.1) and (2.5,0.1) .. (2.5,-0.5);
    \draw (0.5,-0.5) .. controls (0.8,-0.5) and (0.4,-0.1) .. (0.7,-0.1) .. controls (0.8,-0.1) and (2.3,-0.1) .. (2.3,-0.5);
    \end{scope}
    \begin{scope}
      \clip (1.7,-0.3) rectangle (2.1,0.1);
    \fill[yellow!20,draw=none]
    (1.6,-0.5) ..
    controls (1.8,-0.2) and (2.3,0.6) .. (1.4,0.6) .. controls (0.5,0.6) and (1.3,-0.8) .. (1.5,-0.8) .. controls (2.5,-0.8) and (2.5,0.3) .. (2.5,0.3) --
    (3.3,-0.3) .. controls (2.5,0.0) and (2.3,-0.9) .. (1.4,-0.9) .. controls (1.1,-0.9) and (0.4,0.75) .. (1.4,0.75) .. controls (2.4,0.75) and (2.0,-0.2) ..
    (1.8,-0.5) -- (1.6,-0.5);
    \draw (1.6,-0.5) ..
    controls (1.8,-0.2) and (2.3,0.6) .. (1.4,0.6) .. controls (0.5,0.6) and (1.3,-0.8) .. (1.5,-0.8) .. controls (2.5,-0.8) and (2.5,0.3) .. (2.5,0.3);
    \draw (1.8,-0.5) ..
    controls (2.0,-0.2) and (2.4,0.75) .. (1.4,0.75) .. controls (0.4,0.75) and (1.1,-0.9) .. (1.4,-0.9) .. controls (2.3,-0.9) and (2.5,0.0) .. (3.3,-0.3);
    \end{scope}
    \begin{scope}
      \clip(1.5,-1) rectangle (1.9,-0.6);
    \fill[yellow!20,draw=none]
    (1.6,-0.5) ..
    controls (1.8,-0.2) and (2.3,0.6) .. (1.4,0.6) .. controls (0.5,0.6) and (1.3,-0.8) .. (1.5,-0.8) .. controls (2.5,-0.8) and (2.5,0.3) .. (2.5,0.3) --
    (3.3,-0.3) .. controls (2.5,0.0) and (2.3,-0.9) .. (1.4,-0.9) .. controls (1.1,-0.9) and (0.4,0.75) .. (1.4,0.75) .. controls (2.4,0.75) and (2.0,-0.2) ..
    (1.8,-0.5) -- (1.6,-0.5);
    \draw (1.6,-0.5) ..
    controls (1.8,-0.2) and (2.3,0.6) .. (1.4,0.6) .. controls (0.5,0.6) and (1.3,-0.8) .. (1.5,-0.8) .. controls (2.5,-0.8) and (2.5,0.3) .. (2.5,0.3);
    \draw (1.8,-0.5) ..
    controls (2.0,-0.2) and (2.4,0.75) .. (1.4,0.75) .. controls (0.4,0.75) and (1.1,-0.9) .. (1.4,-0.9) .. controls (2.3,-0.9) and (2.5,0.0) .. (3.3,-0.3);
    \end{scope}
    \begin{scope}
      \clip(2.2,-0.6) rectangle (2.6,-0.1);
    \fill[yellow!20,draw=none]
    (-0.5,0.5) .. controls (0.3,0.5) and (0.2,0.1) .. (0.5,0.1) .. controls (0.8,0.1) and (2.5,0.1) .. (2.5,-0.5) .. controls (2.5,-1.2) and (1.4,-1.2) .. (1.6,-0.5) --
    (1.8,-0.5) .. controls (1.6,-1.1) and (2.3,-1.1) .. (2.3,-0.5) .. controls (2.3,-0.1) and (0.8,-0.1) .. (0.5,-0.1) .. controls (0.7,-0.1) and (0.8,-0.5) .. (0.5,-0.5) -- (-0.5,0.5);
    \draw (-0.5,0.5) .. controls (0.3,0.5) and (0.2,0.1) .. (0.5,0.1) .. controls (0.8,0.1) and (2.5,0.1) .. (2.5,-0.5) .. controls (2.5,-1.2) and (1.4,-1.2) .. (1.6,-0.5);
    \draw (0.5,-0.5) .. controls (0.8,-0.5) and (0.4,-0.1) .. (0.7,-0.1) .. controls (0.8,-0.1) and (2.3,-0.1) .. (2.3,-0.5) .. controls (2.3,-1.1) and (1.6,-1.1) .. (1.8,-0.5);
    \end{scope}
    \fill[draw=none, blue!10,opacity=0.3] (0,2) arc [start angle=90, delta angle=-180,x radius=1,y radius=2] -- (3,-2) arc [start angle=270,delta angle=180, x radius=1,y radius=2] -- (0,2);
    \draw[fill=blue!10,opacity=0.3] (0,0) ellipse (1 and 2);
    \draw[dotted] (3,2) arc [start angle=90, delta angle=180, x radius=1, y radius=2];
    \draw (3,2) arc [start angle=90, delta angle=-180, x radius=1, y radius=2];
    \draw (0,2) -- (3,2);
    \draw (0,-2) -- (3,-2);
    \draw[thin,dotted] (0.5,-0.5) -- (-0.5,0.5);
    \draw[thin,dotted] (3.3,-0.3) -- (2.5,0.3);
\end{scope}
\end{tikzpicture}
\caption{Left: The triple $(D^2\times[0,1],[-\tfrac{1}{2},\tfrac{1}{2}]\times[0,1],\{-\tfrac{1}{2},\tfrac{1}{2}\}\times[0,1])$.
The tuple $(D^2\times[0,1],[-\tfrac{1}{2},\tfrac{1}{2}]\times[0,1])$ can be interpreted as a zero-framed 1-tangle of knot type the unknot (the core $\{0\}\times[0,1]$ of $[-\tfrac{1}{2},\tfrac{1}{2}]\times[0,1]$ is an unknot (as a long knot)).
Right: the zero-framed 1-tangle of knot type the figure eight knot.}
\label{fig:tangle}
\end{figure}
In the construction of the canonical surface from a braid word $w$, every band corresponding to $a_{i,j}^{\pm1}$ can be understood as a closed band in this sense. We say a Seifert surface $F'$ arises from $F$ by \emph{tying a zero-framed knot} $K$ into the closed band $S$, if $F'=(F\setminus S)\cup \Psi^{-1}(T_K)$, where $(D^2\times[0,1],T_K)$ denotes a so-called zero-framed 1-tangle of knot type $K$; i.e.~$T_K$ is the image of a smooth embedding $\phi\colon[-\tfrac{1}{2},\tfrac{1}{2}]\times[0,1]\to D^2\times [0,1]$ such that $\phi$ restricts to the identity in a neighboorhood of $[-\tfrac{1}{2},\tfrac{1}{2}]\times\{0,1\}$, $\phi(\{0\}\times[0,1])$ is the knot $K$ (as a long knot), and $\phi(\{-\tfrac{1}{2}\}\times[0,1])$ and $\phi(\{\tfrac{1}{2}\}\times[0,1])$ have algebraic linking number zero; see Figure~\ref{fig:tangle}. Considering only the effect on the link $\partial F$, tying a knot $K$ into $F$ corresponds to a often studied operation in knot theory: a satellite operation with companion $K$ and pattern $\partial F$ viewed in the appropriately framed solid torus that is given as the complement of an open neighborhood of the mantle of the cylinder $Z$. We discuss this in more detail in Section~\ref{sec:proofOfThm}. 

We also note that the construction of tying a knot into a closed band $S$ depends on how $S$ is identified with $[-\tfrac{1}{2},\tfrac{1}{2}]\times[0,1]$ (rotation of 180 degrees on $[-\tfrac{1}{2},\tfrac{1}{2}]\times[0,1]$ amounts to a change in orientation of the knot tied into $S$); however, we will suppress this subtlety, which is justified by the fact that a knot is strongly quasipositive if and only if the same knot with reversed orientation is strongly quasipositive.

\begin{proof}[Proof of Proposition~\ref{prop:main}]
Let $w$ be a $n$-braid word for some positive integer $n$ such that $F$ is isotopic to the canonical surface associated with $w$.
Let $t$ be the number of occurrences of generators $a_{i,j}$ in $w$ 
with negative power, say they are (in order of appearance) $g_1=a_{i_1,j_1}^{-1},g_2=a_{i_2,j_2}^{-1},\ldots, g_t=a_{i_t,j_t}^{-1}$. We set $S_1,\ldots, S_t\subset F$ to be the closed bands corresponding to these generators.
Our strategy is to describe explicitly a strongly quasipositive braid word $v$ that has $P$, the surface obtained from $F$ by tying zero-framed $K_\ell$ into $S_\ell$, as its canonical surface. This is done by replacing  occurrences of $a_{i,j}^{-1}$ in \eqref{eq:wprod}
with strongly quasipositive braid word with associated canonical surface an appropriately framed and knotted band.
Rudolph in \cite{Rudolph_83_ConstructionsOfQP1} showed this is possible for $K_\ell$ being a trefoil.
The novelty of our proof is that we establish that this can be done for any strongly quasipositive knot $K_\ell$ that is not the unknot.

Let $A=A_\ell$ be the zero-framed annulus given as the zero-framing of $K=K_\ell$ (until further notice we suppress the index $\ell$, remember, that the procedure
is performed for all $1\leq \ell\leq n$).
The annulus $A$ is a quasipositive surface; see~\cite[Lemma~1]{rudolph_QPasObstruction}. We recall the short argument: let $Q$ be a quasipositive surface with boundary $K$. Note that a small closed neighborhood of $K=\partial Q$ in $Q$ is a zero-framing of $K$ and, thus, $A$ embeds $\pi$-injective in $Q$. For the latter, we used that $K$ is not the unknot: a neighborhood of the boundary of a connected oriented surface $Q$ with connected boundary $K$ is a full subsurface of $Q$ (i.e.~$\pi_1$-injectively embedded) as long as $Q$ is not a disc. By Rudolph~\cite[Characterization Theorem]{Rudolph_92_ConstructionsIII}, full subsurfaces of quasipositive surfaces are quasipositive surfaces.
\begin{remark}\label{rem:qpbandsvssqpbands}
  The assumption that $K$ is strongly quasipositive and not only quasipositive is subtle, but essential. If $K$ is quasipositive, the annulus $A$ will not be a quasipositive surface in general.
  Otherwise, we could tie a quasipositive smoothly slice knot $K$ into $F$, and by repeating our argument for Theorem~\ref{thm:main1} in the smooth category, we would show that every knot is smoothly concordant to a quasipositive knot, which
  contradicts Proposition~\ref{prop:not_smoothly}.
\end{remark}

Since $A$ is a quasipositive surface, it arises as the canonical surface associated with a strongly quasipositive braid word. Let $w_A=\prod a_{i_k,j_k}$ be such a $n_A$-braid word, for some positive integer $n_A$.
We refer to Figure~\ref{fig:QPA} for an example where $A$ is the zero-framed positive trefoil and $w_A$ is the $6$-braid word $a_{2,6}a_{1,4}a_{2,5}a_{4,6}a_{3,5}a_{1,3}$.
\begin{lemma}\label{lem:twogens}
  We may arrange $w_A$ in such a way that among the generators in $w_A$ exactly two start at strand $1$,
  that is, exactly two generators of the form $a_{1,j}$ occur in $w_A$.
\end{lemma}
\begin{proof}
  In the proof we will actually show more, namely that $w_A$ can be arranged in such a way that for every strand (not only for the first one) there are precisely two generators starting
  or ending at this strand.

  Let $w_A$ be a strongly quasipositive $n_A$-braid word with associated canonical surface $A$.
  Suppose there is a strand $k$ such that precisely one generator in $w_A$ starts or ends at this strand; i.e.~is of form $a_{i,k}$ or $a_{k,j}$. Delete this generator and in all other generators replace $i$ ($j$) with $i-1$ $(j-1)$ whenever $i>k$ ($j>k$) (this corresponds to `deleting the whole strand $k$') producing a strongly quasipositive $(n_A-1)$-braid word. For the associated link and the associated canonical surface, this corresponds to a Markov move; in particular, this operation does not change
  the isotopy type of the canonical surface. Perform this operation finitely many times until we have a strongly quasipositive braid word, which, by abuse of notation, we still denote by $w_A$, and
   such that for all strands $k$ at least two generators in $w_A$ start or end at strand $k$. Since $A$ is an annulus, this implies that for all strands $k$ precisely two generators in $w_A$ start or end at strand $k$, for otherwise the first Betti number of $A$ is at least
  two.
\end{proof}
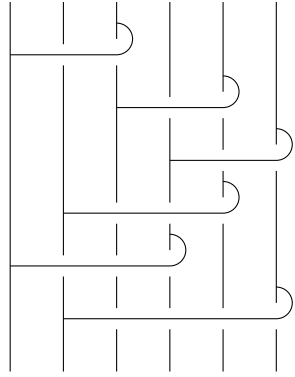
\begin{figure}[h]
  \begin{tikzpicture}
    \begin{scope}[xshift=3cm,scale=0.7]
      \draw(0,0) -- (0,7);
      \draw(1,0) -- (1,1.8) (1,2.2) -- (1,5.8)  (1,6.2) -- (1,7);
      \draw(2,0) -- (2,0.8) (2,1.2) -- (2,1.8) (2,2.2) -- (2,2.8) (2,3.2) -- (2,5.8) (2,6.3) -- (2,7);
      \draw(3,0) -- (3,0.8) (3,1.2) -- (3,1.8) (3,2.3) -- (3,2.8) (3,3.2) -- (3,4.8) (3,5.2) -- (3,7);
      \draw(4,0) -- (4,0.8) (4,1.2) -- (4,2.8) (4,3.3) -- (4,3.8) (4,4.2) -- (4,4.8) (4,5.3) -- (4,7);
      \draw(5,0) -- (5,0.8) (5,1.3) -- (5,3.8) (5,4.3) -- (5,7);
      \draw(0,6) -- (2,6) arc [start angle=270, delta angle=180, radius=0.3];
      \draw(2,5) -- (4,5) arc [start angle=270, delta angle=180, radius=0.3];
      \draw(3,4) -- (5,4) arc [start angle=270, delta angle=180, radius=0.3];
      \draw(1,3) -- (4,3) arc [start angle=270, delta angle=180, radius=0.3];
      \draw(0,2) -- (3,2) arc [start angle=270, delta angle=180, radius=0.3];
      \draw(1,1) -- (5,1) arc [start angle=270, delta angle=180, radius=0.3];
    \end{scope}
  \end{tikzpicture}
  \caption{The zero-framed annulus $A$ of knot type $T_{2,3}$ as the canonical surface associated to the 6-braid word $a_{2,6}a_{1,4}a_{2,5}a_{4,6}a_{3,5}a_{1,3}$.}
\label{fig:QPA}
\end{figure}
Next we modify the word $w_A$, making precise the following idea about its associated canonical Seifert surface $A$: turning the annulus $A$ into a quasipositive band that can replace the corresponding band in $F$.
We can and do assume that $w_A$ satisfies the statement of Lemma~\ref{lem:twogens},
that is, there are exactly two generators in $w_A$ that start at the first strand. Suppose these are $a_{1,j}$ and $a_{1,j'}$ and $a_{1,j}$ occurs first
in the braid word $w_A$.
We build a new strongly quasipositive $(n_A+1)$-braid word $v_A$ as follows.
Replace $a_{1,j}$ by $a_{2,j+1}$, $a_{1,j'}$ by $a_{1,j'+1}$, and all the other generators $a_{i,j}$ in $w_A$ by $a_{i+1,j+1}$.
Figure~\ref{fig:v_A} illustrates $v_A$ (ignore the dotted line in that figure until further notice).
\begin{figure}[h]
  \begin{tikzpicture}
    \begin{scope}[xshift=-3cm,scale=0.7]
      \draw(0,0) -- (0,7.3);
      \draw(1,0) -- (1,1.8) -- (1,2.2) -- (1,5.8)  (1,6.2) -- (1,7.4);
      \draw(2,0) -- (2,1.8) (2,2.2) -- (2,5.8) (2,6.2) -- (2,7.4);
      \draw(3,0) -- (3,0.8) (3,1.2) --  (3,1.8) (3,2.2) -- (3,2.8) (3,3.2) -- (3,5.8) (3,6.3) -- (3,7.4);
      \draw(4,0) -- (4,0.8) (4,1.2) -- (4,1.8) (4,2.3) -- (4,2.8) (4,3.2) -- (4,4.8) (4,5.2) -- (4,7.4);
      \draw(5,0) -- (5,0.8) (5,1.2) -- (5,2.8) (5,3.3) -- (5,3.8) (5,4.2) -- (5,4.8) (5,5.3) -- (5,7.4);
      \draw(6,0) -- (6,0.8) (6,1.3) -- (6,3.8) (6,4.3) -- (6,7.4);
      \draw[thick, densely dotted,blue!50!black] (1,7.4)-- (1,7.6) -- (7,7.6) arc [start angle=270, delta angle=180, radius=0.3] -- (7,7.8) (7,7.4) -- (7,0);
      \draw(0,6) -- (3,6) arc [start angle=270, delta angle=180, radius=0.3];
      \draw(3,5) -- (5,5) arc [start angle=270, delta angle=180, radius=0.3];
      \draw(4,4) -- (6,4) arc [start angle=270, delta angle=180, radius=0.3];
      \draw(2,3) -- (5,3) arc [start angle=270, delta angle=180, radius=0.3];
      \draw(1,2) -- (4,2) arc [start angle=270, delta angle=180, radius=0.3];
      \draw(2,1) -- (6,1) arc [start angle=270, delta angle=180, radius=0.3];
      \draw[ultra thick,green!50!black] (-0.3,0.05) -- (0.3,0.05);
      \draw[ultra thick,green!50!black] (0.7,7.35) -- (1.3,7.35);
    \end{scope}
    \begin{scope}[xshift=3cm, scale=0.4,yshift=3cm]
      \fill[yellow!10,draw=none] (3,0) -- (0,0) -- (0,6) -- (3,6) -- (3,5) -- (1,5) -- (1,1) -- (3,1) -- (3,0);
      \draw(3,0) -- (0,0) -- (0,6) -- (3,6) (3,5) -- (1,5) -- (1,1) -- (3,1);
      \draw[ultra thick,green!50!black] (-0.3,4) -- (1.3,4);
      \draw[thick, purple!50!black,->, decorate, decoration={snake,amplitude=.4mm, segment length=2mm, post length=1mm}](2.8,3) -- (6.5,3);
      \begin{scope}[xshift=2cm]
	\fill[yellow!10,draw=none] (5,0) -- (5,4) -- (9,4)-- (9,3) -- (6,3) -- (6,0) -- (5,0);
\fill[yellow!10,draw=none] (9,1) -- (7,1) -- (7,2.8) -- (8,2.8) -- (8,2) -- (9,2);
\fill[yellow!10,draw=none] (7,4.2) rectangle (8,6);
      \draw(5,0) -- (5,4) -- (9,4) (9,3) -- (6,3) -- (6,0) (9,1) -- (7,1) -- (7,2.8) (7,4.2) -- (7,6) (8,6) -- (8,4.2) (8,2.8) -- (8,2) -- (9,2);
\draw[ultra thick,green!50!black] (4.7,0) -- (6.3,0);
\draw[ultra thick,green!50!black] (6.7,6) -- (8.3,6);
    \end{scope}
     \end{scope}
  \end{tikzpicture}
\caption{Left: the canonical surface associated to 7-braid word $v_A=a_{3,7}a_{2,5}a_{3,6}a_{5,7}a_{4,6}a_{1,4}$ build form the six braid word $w_A=a_{1,4}a_{2,5}a_{4,6}a_{3,5}a_{1,3}a_{2,6}$ (ignoring the dotted and green part).
Right: Cutting open a zero-framed annulus and introducing a negative twist. This is the process that happens on the level of Seifert surfaces from Figure~\ref{fig:QPA} to the left-hand side of this figure.}
\label{fig:v_A}
\end{figure}
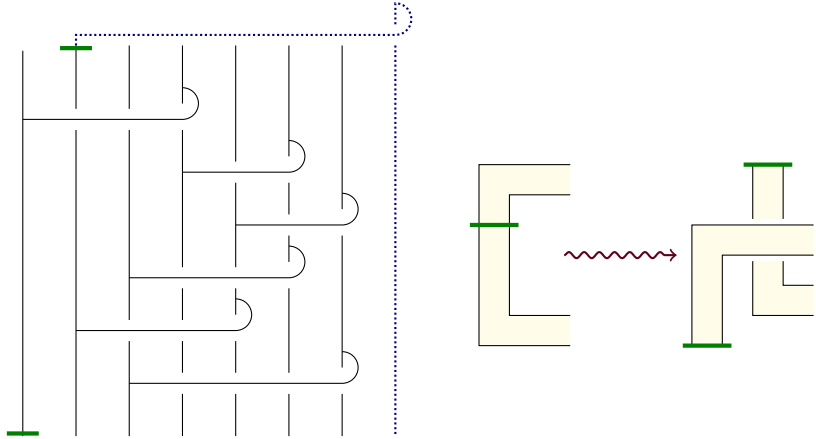
Note that the canonical surface associated to $v_A$ is a disc. The picture to keep in mind is that we are creating a `quasipositive band $F(v_A)$ with a negative twist in it',
by cutting open a  zero-framed quasipositive annulus $F(w_A)$; see the right-hand side of Figure~\ref{fig:v_A}.

The construction of a braid word $v_A$ and the quasipositive band $F(v_A)$ is conducted for $A=A_1,\ldots,A_t$. We denote by $v_1,\ldots,v_t$ the
resulting braid words and by $F_1,\ldots,F_t$ the resulting surfaces. We let $i_1,\ldots,i_t$ denote the number of strands of $v_1,\ldots,v_t$, respectively.

We proceed now by induction on $t$, replacing successively negative generators $g_t,g_{t-1},\ldots,g_1$ in $w$ as follows.
Out of the $n$-braid word $w$ we build a new braid word $w'$ by first replacing all generators $a_{i,j}^{\pm 1}$ except $g_t$ by
\[\left\{\begin{array}{lll}(a_{i,j})^{\pm 1}, & \text{if }& j\leq i_t;\\
(a_{i,j+n_{A_t}-1})^{\pm 1},& \text{if }& i\leq i_t<j;\\
(a_{i+n_{A_t}-1,j+n_{A_t}-1})^{\pm 1},& \text{if }& i_t<i\end{array}\right..\]
Then we replace $g_t$ by the following braid word $\wt{v}_t$: $\wt{v}_t$ is given by shifting all generators of $v_t$ to the right by $i_t-1$ (that is replace $a_{i,j}$ in $v_{n}$ by $a_{i+i_t-1,j+i_t-1}$) and adding the generator $a_{i_t+1,j_t+i_t-1}$.
The procedure is illustrated in Figure~\ref{fig:w}: first the passage from $w$ to $w'$ is described and then
we replace a negative generator in the word $w'$ by a band from Figure~\ref{fig:v_A} (this time including the dotted line).
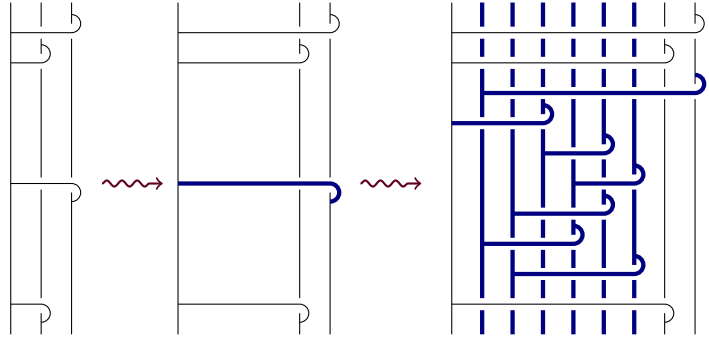
\begin{figure}[h]
\begin{tikzpicture}
  \begin{scope}[scale=0.4,xshift=-2cm]
    \begin{scope}[xshift=-5.5cm]
      \draw(0,0) -- (0,11);
      \draw(1,0) -- (1,0.7) (1,1.2) -- (1,4.8) (1,5.2) -- (1,8.8) (1,9.3) -- (1,9.8) (1,10.2) -- (1,11);
      \draw(2,0) -- (2,4.7) (2,5.2) -- (2,9.8) (2,10.3) -- (2,11);
      \draw(0,10) -- (2,10) arc [start angle=270, delta angle=180, radius=0.3];
      \draw(0,9) -- (1,9) arc [start angle=270, delta angle=180, radius=0.3];
      \draw(0,5) -- (2,5) arc [start angle=90, delta angle=-180, radius=0.3];
      \draw(0,1) -- (1,1) arc [start angle=90, delta angle=-180, radius=0.3];
    \end{scope}
    \draw[thick, purple!50!black,->, decorate, decoration={snake,amplitude=.4mm, segment length=2mm, post length=1mm}](-2.5,5) -- (-0.5,5);
    \draw[thick, purple!50!black,->, decorate, decoration={snake,amplitude=.4mm, segment length=2mm, post length=1mm}](6,5) -- (8,5);
    \begin{scope}[xshift=0cm]
      \draw(0,0) -- (0,11);
      \draw(4,0) -- (4,0.7) (4,1.2) -- (4,4.8) (4,5.2) -- (4,8.8) (4,9.3) -- (4,9.8) (4,10.2) -- (4,11);
      \draw(5,0) -- (5,4.7) (5,5.2) -- (5,9.8) (5,10.3) -- (5,11);
      \draw(0,10) -- (5,10) arc [start angle=270, delta angle=180, radius=0.3];
      \draw(0,9) -- (4,9) arc [start angle=270, delta angle=180, radius=0.3];
      \draw[ultra thick, blue!50!black](0,5) -- (5,5) arc [start angle=90, delta angle=-180, radius=0.3];
      \draw(0,1) -- (4,1) arc [start angle=90, delta angle=-180, radius=0.3];
    \end{scope}
  \end{scope}
  \begin{scope}[scale=0.4,xshift=7cm]
    \draw(0,0) -- (0,11);
    \draw[ultra thick, blue!50!black](1,0) -- (1,0.8) (1,1.2) -- (1,6.8) (1,7.2) -- (1,8.8) (1,9.2) -- (1,9.8) (1,10.2) -- (1,11);
    \draw[ultra thick, blue!50!black](2,0) -- (2,0.8) (2,1.2) -- (2,2.8) (2,3.2) -- (2,6.8) (2,7.2) -- (2,7.8) (2,8.2) -- (2,8.8) (2,9.2) -- (2,9.8) (2,10.2) -- (2,11);
    \draw[ultra thick, blue!50!black](3,0) -- (3,0.8) (3,1.2) -- (3,1.8) (3,2.2) -- (3,2.8) (3,3.2) -- (3,3.8) (3,4.2) -- (3,6.8) (3,7.3) -- (3,7.8) (3,8.2) -- (3,8.8) (3,9.2) -- (3,9.8) (3,10.2) -- (3,11);
    \draw[ultra thick, blue!50!black](4,0) -- (4,0.8) (4,1.2) -- (4,1.8) (4,2.2) -- (4,2.8) (4,3.3) -- (4,3.8) (4,4.2) -- (4,5.8) (4,6.2) -- (4,7.8) (4,8.2) -- (4,8.8) (4,9.2) -- (4,9.8) (4,10.2) -- (4,11);
    \draw[ultra thick, blue!50!black](5,0) -- (5,0.8) (5,1.2) -- (5,1.8) (5,2.2) -- (5,3.8) (5,4.3) -- (5,4.8) (5,5.2) -- (5,5.8) (5,6.3) -- (5,7.8) (5,8.2) -- (5,8.8) (5,9.2) -- (5,9.8) (5,10.2) -- (5,11);
    \draw[ultra thick, blue!50!black](6,0) -- (6,0.8) (6,1.2) -- (6,1.8) (6,2.3) -- (6,4.8) (6,5.3) -- (6,7.8) (6,8.2) -- (6,8.8) (6,9.2) -- (6,9.8) (6,10.2) -- (6,11);
    \draw(7,0) -- (7,0.7) (7,1.2) -- (7,7.8) (7,8.2) -- (7,8.8) (7,9.3) -- (7,9.8) (7,10.2) -- (7,11);
    \draw(8,0) -- (8,7.8) (8,8.3) -- (8,9.8) (8,10.3) -- (8,11);
    \draw(0,10) -- (8,10) arc [start angle=270, delta angle=180, radius=0.3];
    \draw(0,9) -- (7,9) arc [start angle=270, delta angle=180, radius=0.3];
    \draw[ultra thick, blue!50!black](1,8) -- (8,8) arc [start angle=270, delta angle=180, radius=0.3];
    \draw[ultra thick, blue!50!black](0,7) -- (3,7) arc [start angle=270, delta angle=180, radius=0.3];
    \draw[ultra thick, blue!50!black](3,6) -- (5,6) arc [start angle=270, delta angle=180, radius=0.3];
    \draw[ultra thick, blue!50!black](4,5) -- (6,5) arc [start angle=270, delta angle=180, radius=0.3];
    \draw[ultra thick, blue!50!black](2,4) -- (5,4) arc [start angle=270, delta angle=180, radius=0.3];
    \draw[ultra thick, blue!50!black](1,3) -- (4,3) arc [start angle=270, delta angle=180, radius=0.3];
    \draw[ultra thick, blue!50!black](2,2) -- (6,2) arc [start angle=270, delta angle=180, radius=0.3];
    \draw(0,1) -- (7,1) arc [start angle=90, delta angle=-180, radius=0.3];
  \end{scope}
\end{tikzpicture}
\caption{Replacing a negative band by a strongly quasipositive braid word in the proof of Proposition~\ref{prop:main}.}
\label{fig:w}
\end{figure}

By abuse of notation, we call the resulting $(n+i_t-1)$-braid word $w$ and its negative generators $g_1,\ldots g_{t-1}$.
Let $v$ denote the $((n+i_t+\dots+i_1)-t)$-braid word obtained by repeating the process described in the last paragraph $t-1$ more times. Then $v$
is by construction a strongly quasipositive word. Its associated canonical surface is the surface $P$ obtained by replacing $t$ bands with negative twists by quasipositive
bands $F_1,\ldots,F_t$.
\end{proof}

\section{Proof of Theorem~\ref{thm:main2}}\label{sec:proofOfThm}
Take a link $L$ and a Seifert surface $F$ for it; i.e.~a Seifert surface with oriented boundary $L$. By Proposition~\ref{prop:main} there exist an integer $t$ and closed bands $S_1,\ldots,S_t$ in $F$ such that tying any non-trivial strongly quasipositive knots $K_1,\ldots,K_t$ into these bands yields a quasipositive surface. Call this surface $F'$ and let $L'$ be the boundary of $F'$.

We begin with the following observation. Recall that the link $L'\subset S^3$ is a \emph{satellite} with pattern a link $L\subset D^2\times S^1$ and companion a knot $K\subset S^3$, if there is an orientation preserving embedding of  $D^2\times S^1$ into $S^3$ mapping $L$ to $L'$ such that $\{0\}\times S^1$ maps to $K$ and $\{0\}\times S^1$ and $\{1\}\times S^1$ have algebraic linking number 0; compare e.g.~\cite[Section 4.D]{rolfsen_knotsandlinks}).

\begin{proposition}\label{prop:is_satellite}
  The link $L'$ arises from $L$ as an $t$-fold satellite construction.
  Namely, there exists a sequence of links $L_0,\ldots,L_t$ with $L_0=L$, $L_t=L'$ such that for $i=1,\ldots,t$, the link $L_{i}$ is the satellite link with companion knot $K_{i}$ and
  pattern $L_{i-1}$.
\end{proposition}
\begin{proof}
This is a consequence of the fact that tying a knot $K$ into a closed band in a Seifert surface $F$ amounts to realizing the boundary of the resulting surface as a satellite with pattern $\partial F$ and companion $K$. For completeness, we provide a detailed proof.

Set $F_0:=F$ and $L_0:=L$.
  Take $\eta_1$ to be a simple closed curve going once along the band $S_1$; see Figure~\ref{fig:tieup}.
  \begin{figure}[h]
  \begin{tikzpicture}
      \draw[fill=yellow!5!white] (-2,0) rectangle (2,1);
      \draw(0,0.5) node [scale=0.7]{$L_{i-1}$};
      \draw (0.35,-1.15) ellipse [x radius=0.3, y radius=0.7];
      \fill[draw=none,color=white] (-0.3,-1.5) rectangle (0.3,-0.8);
      \draw (0.35,-0.3) node [scale=0.7] {$\eta_i$};
      \draw (-1.8,0) -- (-1.8,-1.3) -- (-0.5,-1.3) (-0.5,-1) -- (-1.5,-1) -- (-1.5,0);
      \draw (1.8,0) -- (1.8,-1.3) -- (0.8,-1.3) (0.8,-1) -- (1.5,-1) -- (1.5,0);
      \begin{scope}
	\clip (-2,-1.5) rectangle (0,-0.5);
	\fill[draw=none,color=blue!20]  (0.5,-1.3) -- (0.25,-1.3) .. controls (0.05,-1.3) and (0.15,-1.3) .. (0.05,-1.2) -- (-0.05,-1.1) .. controls (-0.15,-1) and (-0.05,-1) .. (-0.25,-1) -- (-0.5,-1) -- (-0.5,-1.3) -- (0.5,-1.3);
	\fill[draw=none,color=white] (-0.5,-1.3) -- (-0.25,-1.3) .. controls (-0.05,-1.3) and (-0.15,-1.3) .. (-0.05,-1.2) -- (0.05,-1.1) .. controls (0.15,-1) and (0.05,-1) .. (0.25,-1) -- (0.5,-1) -- (0.5,-1.3) -- (-0.5,-1.3);
      \end{scope}
      \begin{scope}
	\clip (0,-1.5) rectangle (2,-0.5);
	\fill[draw=none,color=blue!20] (-0.5,-1.3) -- (-0.25,-1.3) .. controls (-0.05,-1.3) and (-0.15,-1.3) .. (-0.05,-1.2) -- (0.05,-1.1) .. controls (0.15,-1) and (0.05,-1) .. (0.25,-1) -- (0.5,-1) -- (0.5,-1.3) -- (-0.5,-1.3);
	\fill[draw=none,color=white]  (0.5,-1.3) -- (0.25,-1.3) .. controls (0.05,-1.3) and (0.15,-1.3) .. (0.05,-1.2) -- (-0.05,-1.1) .. controls (-0.15,-1) and (-0.05,-1) .. (-0.25,-1) -- (-0.5,-1) -- (-0.5,-1.3) -- (0.5,-1.3);
      \end{scope}
      \draw (-0.5,-1.3) -- (-0.25,-1.3) .. controls (-0.05,-1.3) and (-0.15,-1.3) .. (-0.05,-1.2) (0.05,-1.1) .. controls (0.15,-1) and (0.05,-1) .. (0.25,-1) -- (0.5,-1);
      \draw (0.5,-1.3) -- (0.25,-1.3) .. controls (0.05,-1.3) and (0.15,-1.3) .. (0.05,-1.2) -- (-0.05,-1.1) .. controls (-0.15,-1) and (-0.05,-1) .. (-0.25,-1) -- (-0.5,-1);
      \draw (-0.25,-1.15) -- ++(-0.25,0.25) -- node [midway,above,scale=0.7] {$S_i$} ++(-0.6,0);
      \begin{scope}[y=0.6pt, x=0.6pt,yscale=-1, inner sep=0pt, outer sep=0pt]

	\begin{scope}[xshift=3cm,yshift=-1cm]
  \path[draw=black]
    (32.6,113) .. controls (44.5593,113.3788) and (45.2692,113.6082) ..
    (51.4494,114.1876)(60.7909,114.4450) .. controls (79.1129,112.1593) and
    (117.8830,120.7822) .. (129.8930,107.5059) .. controls (138.1716,98.3544) and
    (124.8307,90.4398) .. (117.4597,90.5888)(105.5662,92.0588) .. controls
    (87.1034,93.9450) and (55.8026,98.5753) .. (55.5921,114.3213) .. controls
    (55.3575,131.8609) and (80.2237,128.2143) ..
    (87.3366,116.8277)(96.4238,108.5424) .. controls (126.2156,77.6369) and
    (124.2907,77.5949) .. (172,77);

    \path[draw=black]
    (32.6,109) .. controls (44.5590,109.4193) and (44.4670,110.0496) ..
    (50.6473,110.6290)(59.9887,110.8864) .. controls (78.3108,108.6007) and
    (113.7109,118.4288) .. (126.4182,105.8182) .. controls (134.7069,97.5926) and
    (125.8692,92.7186) .. (113.8512,93.1774)(107.3031,88.0993) .. controls
    (88.8404,89.9855) and (55.9966,96.1303) .. (53.3203,111.6486) .. controls
    (50.0121,130.8304) and (77.5510,138.4108) ..
    (90.1429,116.5605)(91.7462,108.4583) .. controls (121.5380,77.5528) and
    (124.9586,74.3036) .. (172,74.5);
    \draw (32.6,113) -- (12,113) -- (12,90);
    \draw (32.6,109) -- (16,109) -- (16,90);
    \draw (12,70) -- (12,63) -- (16,57) -- (16,50);
    \draw (15,70) -- (15,63) -- (14.5,62) (12.5,58) --  (12,57) -- (12,50);
    \draw (172,77) -- (192,77) -- (192,50);
    \draw (172,74.5) -- (188,74.5) -- (188,50);
    \draw[fill=yellow!5] (0,50) rectangle (200,0);
    \fill[white,draw=none] (-15,70) rectangle (65,90);
    \draw[blue,dashed] (-15,70) rectangle (65,90);
    \draw (25,80) node [scale=0.7] {$-3$ full twists};
    \draw (100,27) node [scale=0.7] {$L_{i-1}$};
\end{scope}
\end{scope}
\end{tikzpicture}
\caption{Illustration of the proof of Proposition~\ref{prop:is_satellite}.  Left: the band $S_i$ and the circle $\eta_i$. The link $L_{i-1}$ can be regarded as a link in the complement of a neighborhood
of $\eta_i$. Right: the band $S_i$ is replaced by a band that is tied along the knot $K_i$. For clarity we drew $K_i$ to be a trefoil;
in our application the knot $K_i$ will have Alexander polynomial one in addition to being non-trivial and strongly quasipositive. Note the three negative twists to keep the framing of the inserted band
equal to zero.}\label{fig:tieup}
\end{figure}
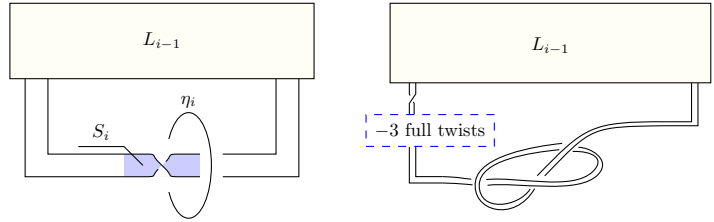
More precisely, given a cylinder $Z_1\overset{\Psi}{\cong} D^2\times[0,1]$ defining the closed band $S_1$ (as in the definition of a closed band in Section~\ref{sec:proofOfProp}), take $\eta_1$ to be $\Psi^{-1}((\partial D^2)\times\{\frac{1}{2}\})$.
  The link $L_{0}=L$ can be considered as a link in a solid torus $S^3\setminus\nu(\eta_{1})$,
  where $\nu(\cdot)$ denotes `an open tubular neighborhood of $\cdot$'. Identifying $S^3\setminus\nu(\eta_{1})$ with $D^2\times S^1$ (orientation preserving, mapping the null-homologous longitude\footnote{Here we take a longitude to be (the homology class of) a curve in the boundary of $S^3\setminus\nu(\eta_{1})$ that is null-homologous in the closure of $\nu(\eta_{1})$ and links $\eta_{1}$ once positively.}
  to $\{1\}\times S^1$) further allows to view $L=L_0$ as a link in $D^2\times S^1$.

  Denote the resulting surface of tying the zero-framed knot $K_1$ into the closed band $S_{1}\subset F_0=F$ by $F_1$ and its boundary by $L_1$. With this setup it is evident that $L_1$ a satellite with pattern $L_{0}\subset D^2\times S^1$ and companion $K_1$; that is, $L_1$ arises as the image of $L_0\subset D^2\times S^1$ under an orientation preserving embedding of $D^2\times S^1$ into $S^3$ such that $\{0\}\times S^1$ maps to $K_1$ and $\{0\}\times S^1$ and $\{1\}\times S^1$ have algebraic linking number 0.
  Note that the $S_j$ for $j\geq 2$ also constitute closed bands in $F_1$.

  We proceed inductively: define $\eta_i$, $F_i$, and $L_i$ correspondingly (see Figure~\ref{fig:tieup}), observe (as in the case of $i=1$) that $L_i$ arises as the satellite link with pattern $L_{i-1}\subset D^2\times S^1$ and companion $K_i$, and note that $S_j$ for $j\geq i+1$ also constitute closed bands in $F_i$.

\end{proof}

The last ingredient in proving Theorem~\ref{thm:main2} is a result that controls the behavior of concordance after a satellite construction. Apart from the statements about the fundamental group, this is well-known to experts; for example, this is the fact that insures that satellite constructions are defined on the concordance group rather than just on isotopy classes of knots.

\begin{proposition}\label{prop:concordance_satellite}
  Suppose $K\subset S^3$ is a knot, $U$ is the unknot and $A\subset S^3\times[0,1]$ is a topological concordance between them, i.e.~a locally flat annulus such that $A\cap S^3\times\{0\}=K$, $A\cap S^3\times\{1\}=U$. Let $L$ be a pattern
  with $m$ components in the solid
  torus $D^2\times S^1$ and let $L(K)$ and $L(U)$ be corresponding satellites with companion $K$ and $U$, respectively.

  Then there exists a disjoint union of $m$ locally flat annuli $
  L(A)\subset S^3\times[0,1]$ such that
  \[L(A)\cap S^3\times\{0\}=L(K)\times\{0\}\et L(A)\cap S^3\times\{1\}=L(U)\times\{1\},\] where the orientations induced by $L(A)$ on $L(K)\times\{0\}$ and $L(U)\times\{1\}$ are the orientation of $L(K)$ and the reverse orientation of $L(U)$, respectively.

  If additionally the inclusion induced maps
  \[\pi_1(S^3\setminus U)\to \pi_1(S^3\times[0,1]\setminus A) \et \pi_1(S^3\setminus K)\to \pi_1(S^3\times[0,1]\setminus A)\] are an isomorphism and a surjection, respectively, then the inclusion induced maps \[\pi_1(S^3\setminus L(U))\to \pi_1(S^3\times[0,1]\setminus L(A))\et\pi_1(S^3\setminus L(K))\to \pi_1(S^3\times[0,1]\setminus L(A))\] are an isomorphism and a surjection, respectively.
\end{proposition}
\begin{proof}
  The first part of the statement follows by `applying the satellite construction to $A$'. The second part is a Seifert--van Kampen calculation. For the reader's convenience, we provide a detailed proof.

  The $m$ disjoint annuli $L(A)$ are constructed via a satellite operation as follows. Since $A$ is locally flat, there exists
  a tubular neighborhood $V(A)\cong T\times[0,1]$ of $A$ in $S^3\times [0,1]$
  , where $T$ is the open solid torus $(D^2)^\circ\times S^1$; for example, $V(A)$ can be taken to be an open disc subbundle of the normal bundle of $A$ in the sense of~\cite[Section~9.3]{FreedmanQuinn_90_TopOf4Manifolds}. Let $h\colon (T\times[0,1],(\{0\}\times S^1)\times[0,1]) \to (V(A), A)$ be homomorphism parametrizing $V(A)$ such that
  $\lambda=\{\frac{1}{2}\}\times S^1\subset T$ gets mapped to a zero-framing of $K$ and $U$, respectively (i.e.~a curve that has algebraic linking number 0 with $K$ and $L$, respectively). To arrange this note that any such homeomorphism $h$ necessarily maps $\lambda\times\{0\}$ and $\lambda\times\{1\}$ to the same framing of $K$ and $U$, respectively. Thus, choosing that $h$ maps $\lambda\times\{1\}$ to the zero-framing of $U$ (e.g.~by precomposing with a self-homeomorphism of $T$) automatically arranges that $\lambda\times\{0\}$ maps to the zero-framing of $K$. To be conform with our smooth setup for links in $S^3$, we also arrange that $h$ restricts to a smooth embedding on $T\times\{0,1\}$.

 Define $L(A)=h(L\times[0,1])$. It is clearly a topological concordance between $L(K)$ and $L(U)$ since $h(V(A))$ establish (local) flatness. Note that
 here the correct arrangement of the framing is used. 

  To prove the isomorphism and surjection on fundamental groups,
  we let $V(K)$ and $V(U)$ denote tubular neighborhoods of $K$ and $U$ respectively and consider the following diagram.
  \[
    \newxycolor{darkgreen}{0.0 0.5 0.0 rgb}
    \newxycolor{darkblue}{0.0 0.0 0.5 rgb}
    \xymatrix@R-=0.2cm{%
      &\pi_1(S^3\setminus V(K))\ar[rd]\ar@/_5pc/@{.>>}@[darkgreen][ddd]^(0.70){{\color{green!50!black}\!\pi_K}}&\\
      \pi_1(\partial V(K))\ar[ur]\ar[rd]\ar@{.>}[ddd]^{\simeq}&&\pi_1(S^3\setminus L(K))\ar[ddd]^{\pi_{L(K)}}\\
							      &\pi_1(V(K)\setminus L)\ar[ru]\ar@/^5pc/@{-->}@[darkblue][ddd]^(0.65){{\color{blue!50!black}\simeq\!}}&\\
      &\pi_1(S^3\times[0,1]\setminus V(A))\ar[rd]&\\
      \pi_1(\partial V(A))\ar[ur]\ar[rd]&&\pi_1(S^3\setminus L(A))\\
					&\pi_1(V(A)\setminus L(A))\ar[ru]&\\
					&\pi_1(S^3\setminus V(U))\ar[rd]\ar@/^5pc/@{.>}@[darkgreen][uuu]^(0.65){{\color{green!50!black}\simeq}}_(0.75){{\color{green!50!black}\pi_U}}&\\
      \pi_1(\partial V(U))\ar[ur]\ar[rd]\ar@{.>}[uuu]_{\simeq}&&\pi_1(S^3\setminus L(U))\ar[uuu]_{\pi_{L(U)}}\\
							      &\pi_1(V(U)\setminus L)\ar[ru]\ar@/_5pc/@{-->}@[darkblue][uuu]_(0.65){{\color{blue!50!black}\simeq}}&\\
    }
  \]
  All the maps are induced by embeddings. Therefore the diagram is commutative.
  The three diamonds are pushout diagrams by the Seifert--van Kampen theorem. The vertical maps to the left are isomorphism by the definition, the curved dashed maps are isomorphisms
  by the construction of $L(A)$. The curved dotted maps $\pi_U$ and $\pi_K$ are an isomorphism and a surjection, respectively, by the assumptions of the proposition. By the universal
  properties of the pushout, we have that the right vertical maps $\pi_{L(U)}$ and $\pi_{L(K)}$ are, respectively, an isomorphism and a surjection.
\end{proof}
We conclude the proof of Theorem~\ref{thm:main2}.

\begin{proof}[Proof of Theorem~\ref{thm:main2}]
Choose a Seifert surface for the link $L$ and a braid word $w$ that has $F$ as associated canonical surface; in particular, the closure of the braid specified by $w$ is $L=\partial F$. Choose a non-trivial
strongly quasipositive knot $K$ with Alexander polynomial one; for a concrete example one can take a Whitehead double of the trefoil; see Section~\ref{sec:overview}.
By Proposition~\ref{prop:main} there are closed bands in $F$ such that, by tying the knot $K$ into these bands, we obtain from the surface $F$ a quasipositive surface $F'$ with boundary a strongly quasipositive link $L'$.
By Proposition~\ref{prop:is_satellite} the link $L'$ arises from $L$ by successive satellite operations with companion knot $K$. Since $K$ has Alexander polynomial one, there exists
a locally flat annulus $A\subset S^3\times[0,1]$ such that $\partial A=K\times\{0\}\cup U\times\{1\}$, $\pi_1(S^3\times[0,1]\setminus A)=\Z$~\cite[Theorem~1.13]{Freedman_82_TheTopOfFour-dimensionalManifolds} and, thus the maps $\pi_1(S^3\setminus U)\to \pi_1(S^3\times[0,1]\setminus A)$
and $\pi_1(S^3\setminus K)\to \pi_1(S^3\times[0,1]\setminus A)$ are an isomorphism and a surjection, respectively. Theorem~\ref{thm:main2} follows from successive
applications of Proposition~\ref{prop:concordance_satellite}.
\end{proof}
\def\cprime{$'$}

\end{document}